\documentclass{article}
\usepackage[T1]{fontenc}
\usepackage[utf8]{inputenc}
\usepackage{graphicx}

\usepackage{amsthm}

\usepackage{authblk}

\usepackage{polski}
\usepackage[english]{babel}
\usepackage{amsfonts}
\usepackage{amssymb}
\usepackage{caption}
\newtheorem{theor}{Theorem}
\newtheorem{con}{Conjecture}
\newtheorem{defin}{Definition}
\providecommand{\keywords}[1]{\textbf{\textbf{Keywords}} #1}
\begin{document}

\title{Fractional and $j$-fold colouring of the plane
}


\author[1,2]{Jarosław Grytczuk \thanks{
Supported by Polish National Science Centre grant 2011/03/B/ST6/01367}}
\author[2]{Konstanty Junosza-Szaniawski }
\author[2]{Joanna Sokół}
\author[2]{Krzysztof Węsek}


\affil[1]{Faculty of Mathematics and Computer Science, Jagiellonian University, Poland\\ 
email: {grytczuk@tcs.uj.edu.pl}
}
\affil[2]{    Faculty of Mathematics and Information Science, Warsaw University of Technology, Poland\\ 

email: \{k.szaniawski, j.sokol, k.wesek\}\@mini.pw.edu.pl
}

\maketitle

\begin{abstract}

We present results referring to the Hadwiger-Nelson problem which asks for the minimum number of colours needed to colour the plane with no two points at distance $1$ having the same colour. Exoo considered a more general problem concerning graphs $G_{[a,b]}$ with $\mathbb{R}^2$ as the vertex set and two vertices adjacent if their distance is in the interval $[a,b]$. Exoo conjectured $\chi(G_{[a,b]}) = 7$ for sufficiently small but positive difference between $a$ and $b$. We partially answer this conjecture by proving that $\chi(G_{[a,b]}) \geq 5$ for $b > a$.

A $j$-fold colouring of graph $G = (V,E)$ is an assignment of $j$-elemental sets of colours to the vertices of $G$, in such a way that the sets assigned to any two adjacent vertices are disjoint. The fractional chromatic number $\chi_f(G)$ is the infimum of fractions $k/j$ for $j$-fold colouring of $G$ using $k$ colours. We generalize a method by Hochberg and O'Donnel (who proved that  $G_{[1,1]} \leq 4.36$) for fractional colouring of graphs $G_{[a,b]}$, obtaining a bound dependant on $\frac{a}{b}$.

We also present few specific and two general methods for $j$-fold coloring of $G_{[a,b]}$ for small $j$, in particular for $G_{[1,1]}$ and $G_{[1,2]}$. The $j$-fold colouring for small $j$ has strong practical motivation especially in scheduling theory, while graph $G_{[1,2]}$ is often used to model hidden conflicts in radio networks.

\keywords{fractional colouring \and Hadwiger-Nelson problem\and colouring of the plane}

\textbf{Mathematics Subject Classification}{ 05C15,  05C10, 05C62}
\end{abstract}

\section{Introduction}

\subsection{Mathematical context}
The famous Hadwiger-Nelson problem asks for the minimum number of colours required to colour the Euclidean plane $\mathbb{R}^2$ in such a way that no two points at distance 1 from each other have the same colour. The question can be equivalently stated in the graph theory language: Recall that a \textit{colouring} of a graph $G=(V,E)$ is a function $c:V \rightarrow K$ (where $K$ is a set of colours) such that every two adjacent vertices  $x,y$ satisfy $c(x)\neq c(y)$. The \textit{chromatic number} $\chi(G)$ of $G$ is the minimal cardinality of the set of colours to colour $G$. Therefore the Hadwiger-Nelson problem is about determining the chromatic number of the graph on the set of vertices $\mathbb{R}^2$ with vertices in distance one adjacent - it is called \textit{unit distance graph} and in this article will be denoted by $G_{[1,1]}$ (the notation will be justified shortly).

The problem was originally proposed by Edward Nelson in 1950 and was made publicly known by Hugo Hadwiger \cite{fisher}. Pionners on the topic observed the following bounds: Nelson first showed that at least $4$ colours are needed (see the proof by Mosers \cite{moser} which uses the so-called Moser spindle) and John Isbell was first to prove that $7$ colours are enough (this result was published by Hadwiger \cite{fisher}). Somehow surprisingly, the aforementioned bounds remain unchanged since their origin in 1950's, i.e. for more than 60 years nobody has found nothing sharper than $4 \leq \chi(G) \leq 7$ - as long as we consider the full generality. Nevetheless, advanced studies on the question, its subproblems and other related topics provide some understanding. For example, if we consider only measurable colourings (i.e. with measurable colours) then at least $5$ colours are necessary \cite{falconer} and if colouring of the plane consists of regions bounded by Jordan curves then at least $6$ colours are required  \cite{woodall}. Generally, across the decades the Hadwiger-Nelson problem inspired many interesting results in the touchpoint of combinatorics and geometry,  a vast number of challenging problems and various applications.

One of possible ways to generalise the first question was presented by  Exoo \cite{exoo}. He considered graphs on the set of vertices $\mathbb{R}^2$ with vertices in distance in the interval $[a,b]$, we denote such graphs by $G_{[a,b]}$. How many colours would be enough to colour such a graph, depending on $a$ and $b$? Are there some important ranges of those parameters? 

The other path for research leading from the Hadwiger-Nelson problem concerns different models of graphs colouring. In the majority of this article we investigate fractional colourings, in some sense a generalisation of the classic colouring: every vertex gets a $j$-elemental set of colours from the set of colours (of size $k$) and the sets for adjacent vertices have to be disjoint. The 'quality' of such a colouring is measured by the fraction $k/j$ and the fractional chromatic number is the infimum of such fractions (one can also ask for the best fraction with a fixed $j$). It can be seen as in this case every single colour is just a $1/j$-size part of a 'complete colour', so we divide 'complete colours' from the classic colouring to somehow save a bit by combining partitioned colours in a clever way. It turns out that in fact we can save much using fractional colouring: fractional chromatic number is always lower or equal to chromatic number of a given graph, but the difference can be arbitrarily large.

The fractional chromatic number of the graph $G_{[1,1]}$ has been studied in the literature. The best upper bound is due to Hochbeg and O'Donnell \cite{hoch-odon} (based on an idea by Fisher and Ullman \cite{fish-ul} of looking for a dense subset of the plane which avoids unit distance) and the best lower bound (using a finite sugraph of the plane) can be found in the book of Scheinerman and Ullman \cite{fgt} (alongside with a good explanation of the upper bound): $3.555\leq \chi_f(G_{[1,1]})\leq 4.36$. Note that the upper bound is much smaller than the upper bound for $\chi(G_{[1,1]})$.



\subsection{Practical motivation}
Colouring of such geometrical graphs has also some significant practical motivations in telecommunication. We will briefly describe an example. More on this topic can be found in the paper of Walczak and Wojciechowski \cite{walczak}.

Consider the following problem: We are given a set of transmitters with equal ranges placed in some area (assume that ranges are equal to $1$) - some of them are in each other range, and some of them are not. If two transmitters are in each other's range, we assume that they can 'quickly' agree on their communication (there are algorithms for it). If two transmitters are not in each other's range but have a common neighbour $C$, then it is possible that $A$ and $B$ would try to transmit to $C$ in the same time - in this case $C$ cannot listen to both messages. Hence we have a conflict which cannot be solved by a direct communication between $A$ and $B$. If two transmitters $A$ and $B$ are not in each other's range and do not have a common neighbour, then there is no conflict. The problem is to assign time-slots for transmitters (in an equitable way) such that no two conflicted transmitters share a time interval.

How we can use graph colouring in this problem? We can create the graph $G$ of conflicts for this network of transmitters: vertices correspond to transmitters and two vertices are adjacent if the corresponding transmitters $X$, $Y$ are at distance grater than $1$ and have a common transmitter in their respective ranges (note that it is possible only if $X$ and $Y$ are at distance in the set $(1,2]$). Hence, using a $k$-colouring of $G$, such that each colour corresponds to one time-slot is one of the possible ways of constructing a proper transmission-schedule. By the definition of colouring, the produced schedule does not contain any pair of conflicted transmitters sharing their time of transmission, and every transmitter gets the same amount of time in one cycle of transmission. The length of the schedule is $k$. On the other hand, we can make use of a fractional colouring of $G$. Since we demand that in one cycle of transmission every transmitter gets a unit of time, then every colour in a fractional coloring corresponds to interval of length $1/j$ of the time unit. Both conflict-freeness and equitability are satisfied in this colouring model also, and (if we used $k$ colours) the length of the schedule is $k/j$. The relation between considered colouring models imply that fractional colouring can possibly produce a shorter schedule in comparison to classic colouring, which (in real world) means: we can save a bit of time. However, the technical constraints imply that too big values of $j$ are not accepted - too fragmented schedule is not practical. This issue plays important role in our work - we devote a part of this article to fractional colourings with 'small $j$'.

Now, consider a graph on the set of vertices $\mathbb{R}^2$ with vertices at distance in the set $[1,2]$ adjacent, which will be denoted in this article as $G_{[1,2]}$. Clearly, every conflict graph of a transmitters network as described above is a subgraph of $G_{[1,2]}$. Therefore, any colouring or fractional colouring of infinite $G_{[1,2]}$ induce, respectively, a colouring or a fractional colouring of a given finite conflict graph. Hence, we get a universal scheme for any network. Additionally, this universal scheme works also if the transmitters are placed on moving vehicles (thus changing the conflict graph). This is an indisputable advantage over using standard colouring algorithms for a given conflict graph.

\section{Preliminaries}
First, we need to formally define fractional and classic colouring of graphs. 
\begin{defin}
\textnormal{Colouring of a graph} $G=(V,E)$ with $k$ colours (or \textnormal{$k$-colouring}) is an assignment of colours $\lbrace1,2, ... ,k\rbrace$ to the vertices of $G$ such that no two adjacent vertices have the same colour.
The smallest number of colours needed to colour a graph $G$ is called \textnormal{chromatic number} and denoted by $\chi(G)$.
\end{defin}
\begin{defin}
Let $P_j(k)$ be a family of all $j$-elemental subsets of $\{1,2,...,k\}$.\\
\textnormal{$j$-fold colouring of graph} $G=(V,E)$ is an assignment of $j$-element sets of colours to the vertices of $G$, i.e. $f: V\longrightarrow P_j(k),$\ such that for any two adjacent vertices $v, w\in V$ we have $f(v)\cap f(w)=\emptyset$.\\
The smallest number of colours $k$ needed for a $j$-fold colouring of a graph $G$ is called the \textnormal{$j$-fold chromatic number} and denoted by $\chi_j(G)$.

The \textnormal{fractional chromatic number} is defined to be:\
\begin{displaymath}\chi_f(G):=\inf_{j\in\mathbb{N}}\frac{\chi_j(G)}{j} =\lim_{j\rightarrow\infty}\frac{\chi_j(G)}{j} \end{displaymath}

\end{defin}
Since an $1$-fold colouring of a graph is just a classic colouring, then $\chi_f(G)\leq\chi(G)$.

Now, we will present two possible, equivalent notions of graphs on the Euclidean plane we are considering in this work. The first is due to Exoo \cite{exoo}. The second is introduced in this paper and is  more convenient for our work, except the Section \ref{s_jarek}. 
\begin{defin}
An $\varepsilon$-unit distance graph denoted by $G_\varepsilon$ is a graph whose vertices are all points of the plane, in which two points are adjacent if their distance $d$ satisfies $1-\varepsilon \leq d \leq 1+\varepsilon$, i.e.\\
$G_\varepsilon=(\mathbb{R}^2, \{ \{x,y\}\subset \mathbb{R}^2\ |\ 1-\varepsilon\leq dist(x,y)\leq 1+\varepsilon  \})$
\end{defin}
\begin{defin}
A graph $G_{[a,b]}$ is a graph whose vertices are all points of the plane $V=\mathbb{R}^2$, in which two points are adjacent if their distance $d$ satisfies $a \leq d \leq b$.\\
$G_{[a,b]}=(\mathbb{R}^2, \{ \{x,y\}\subset \mathbb{R}^2\ |\ a\leq dist(x,y)\leq b  \}$
\end{defin}

Note that every $G_\varepsilon$ graph can be defined as $G_{[1-\varepsilon,1+\varepsilon]}$ and every $G_{[a,b]}$ graph can be defined as $G_\varepsilon$ graph with $\varepsilon=\frac{b-a}{b+a}$. Additionally, it is enough to consider graphs $G_{[1,b]}$ since $G_{[a,b]} \cong G_{[1,b/a]}$.\\

\section{Colouring of $G_{[a,b]}$}\label{s_jarek}

The classic Hadwiger-Nelson problem is considered to be very difficult, in particular giving a better general lower bound than $4$. In his article Exoo \cite{exoo} looked for values of $\varepsilon$ such that we can determine the chromatic number of $G_{\varepsilon}$ and he succeed in finding some such values. He proved that for $0.134756...<\varepsilon<0.138998...$ the exact value of $\chi(G_{\varepsilon})$ is $7$, and for $\varepsilon>0.008533...$ we get $\chi(G_{\varepsilon}) \geq 5$.

His work (including computational experiments) suggested that for small enough $\varepsilon$ the exact value of the chromatic number of $G_\varepsilon$ is $7$.

\begin{con}\textnormal{\cite{exoo}}
For any $\varepsilon>0$ we have $\chi(G_\varepsilon)=7$.
\end{con}

We will give a partial answer to this conjecture. It appears that the pure positivity of $\varepsilon$ allow us to establish a lower bound of $5$, strictly better than $4$. We will use a result by Nielsen \cite{triangles} and for that we need some additional notions.

Given a colouring of the plane $F$, a triangle $T=xyz$ is a \emph{monochromatic limit triangle} if there is a monochromatic set $\{x_1,y_1,z_1,x_2,y_2,z_2,...   \}$ such that $x_n\longrightarrow x$, $y_n\longrightarrow y$, $z_n\longrightarrow z$ and each of the triangles $T_n=x_ny_nz_n$ is similar to $T$.

We will say that triangles $xyz$ and $x'y'z'$ are $\varepsilon -close$ if $\Vert x-x'\Vert<\varepsilon$, $\Vert y-y'\Vert<\varepsilon$ and $\Vert z-z'\Vert<\varepsilon$.

\begin{theor}\label{trojkaty}\textnormal{\cite{triangles}}
Let $F$ be a two-colouring of the plane and let $T$ be a triangle. Then $F$ admits a monochromatic limit triangle congruent to $T$.
\end{theor}

\begin{theor}
For any $\varepsilon>0$ we have $\chi(G_\varepsilon)\geq5$.
\end{theor}
\begin{proof}
Let $\varepsilon>0$ and suppose that $\chi(G_\varepsilon)\leq4$. Let $F$ be a $4$-colouring of $G_\varepsilon$ and let $F'$ be a $2$-colouring of the plane such that each point of colour $1$ or $2$ in $F$ is white in $F'$ and each point of colour $3$ or $4$ in $F$ is black in $F'$. Let $T=xyz$ be an equilateral triangle with a side length 1.\\
From Theorem \ref{trojkaty} $F'$ admits a monochromatic limit triangle congruent to $T$(lets say it is white). So from the definition of monochromatic limit triangle there exists $n$ such that the triangle $T_n=x_ny_nz_n$ is $\frac{\varepsilon}{2} -close$ to $T$. The side lengths of $T_n$ are within the interval $[1-\varepsilon,1+\varepsilon]$, hence the sides of $T_n$ are edges of $G_\varepsilon$. Since all vertices of $T_n$ are white in $F'$ then each of them is coloured with $1$ or $2$ in $F$. Since there are three vertices and two colours there is a monochromatic edge, which contradicts the assumption that $F$ is a $4$-colouring of $G_\varepsilon$.
\qed\end{proof}

We believe that there is still place to prove a better bound.

\section{Fractional chromatic number of $G_{[a,b]}$}

\indent  The best known way to find an upper bound for the fractional chromatic number of $G_{[1,1]}$ is by looking for a set in the plane with highest possible density but without vertices in distance equal to 1. The densest such set was presented by Hochbeg and O'Donnell \cite{hoch-odon}. They presented a proof of the upper bound by a limit argument. In the following theorem we generalize their construction for $G_{[1,b]}$. We give a complex description of the method for any $b$: not only the upper bound, but also an explicit specification of a sequence of fractional colourings "converging" to the upper bound. This sequence give a finite fractional coloring with "quality" as close to the upper bound as needed.

\begin{theor}\label{twplacki}
If $b\ge 1$ then $\chi_{f}(G_{[1,b]})\leq \frac{\sqrt{3}}{3}\cdot\frac{b+
\sqrt{1-x^2}}{x}$ where $x$ is the root of\\ $ bx=\frac{\pi}{6}-\arcsin(x)$.
Moreover there exists a sequence of $(\frac{n}{2(b+1)}-1)^2$-fold colourings with $n^2$ colours for $n\geq1$.
\end{theor}

\begin{proof}
Note: In order to keep this proof shorter we omit some details of calculations.

We will define a set $S$ in the plane with high density but without vertices in distance in $[1,b]$ in the following way. Let set $A$ be an intersection of a disk of unit diameter and a hexagon with common center point as in the Figure \ref{placek}.
\begin{figure}[h]
\center
\includegraphics[scale=0.5]{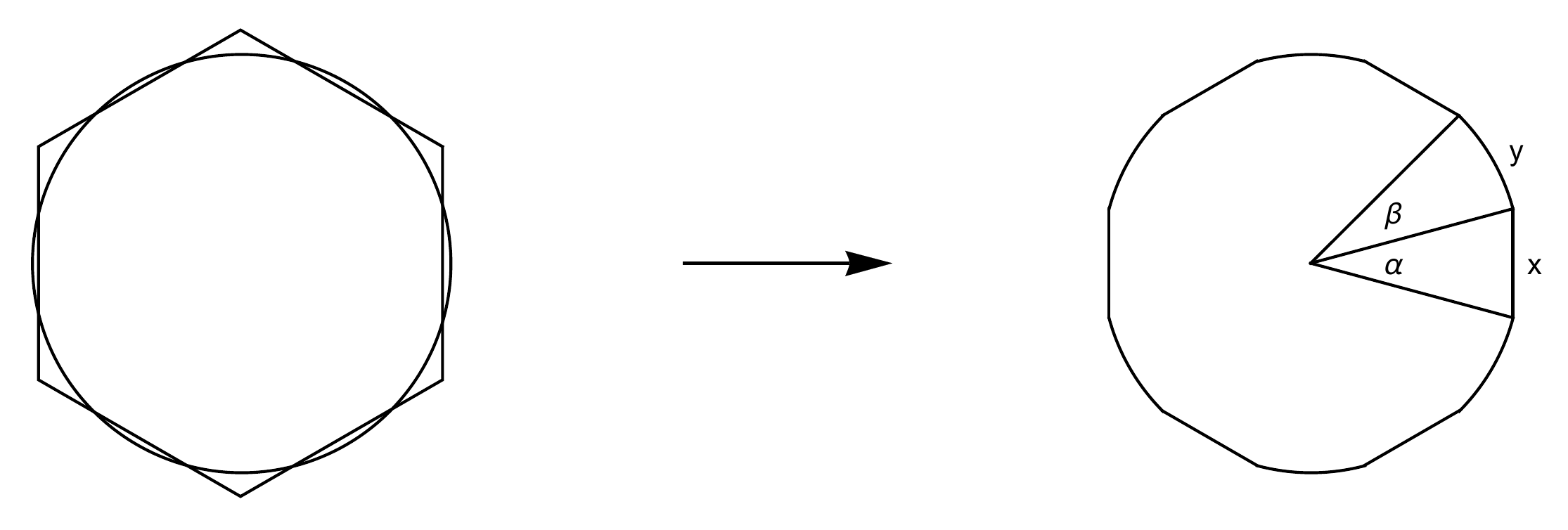}
\caption{Set $A$}\label{placek}
\end{figure} The shape of $A$ changes when we change the size of the hexagon. In fractional colouring of $G_{[1,b]}$ we choose $A$ the ratio between the length of the circular arc - $y$ and the segment - $x$ to be $b$ (it will be explained later). 

From the construction we get the following equalities:
$x=\sin(\frac{\alpha}{2})$, $y=\frac{\beta}{2}$, $\alpha+\beta=\frac{\pi}{3}$. Therefore we obtain
\begin{equation}\label{xy}
y=\frac{\pi}{6}-\arcsin(x)
\end{equation}

Then we build $S$ by placing copies of $A$ on the plane with distance $b$ like in the Figure \ref{placki}. Assume that two neighbouring components of $S$ are centered at $(0,0)$ and $(s,0)$ which determines the location of all components of $S$.
\begin{figure}[h]
\center
\includegraphics[width=12cm]{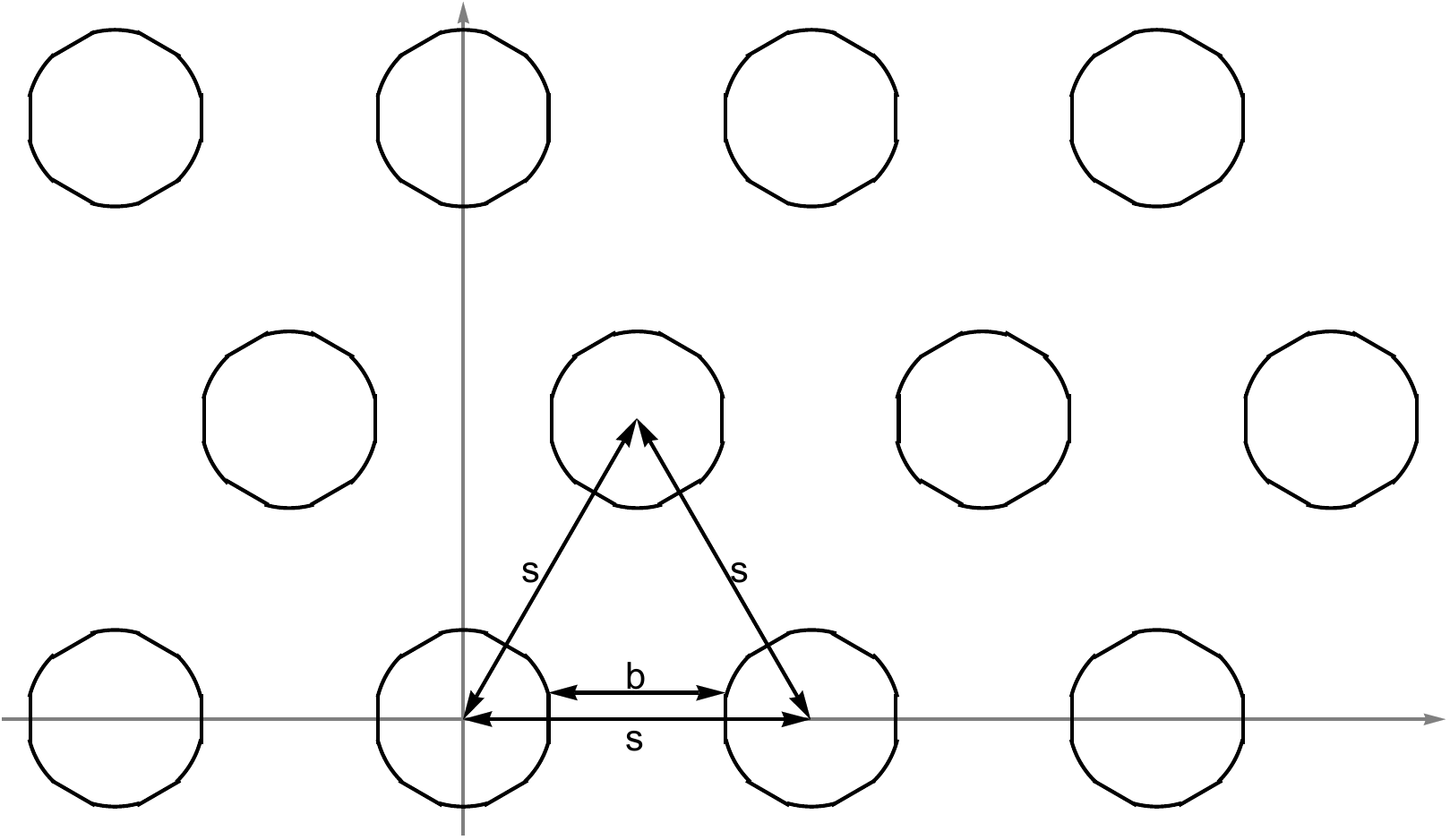}
\caption{Set $S$}\label{placki}
\end{figure}
Let $n$ be a positive integer. We define a tiling with hexagons of width $s/n$ in such a way that the left side of some hexagon is a part of the left vertical segment of the border of $A$. 

For $0\leq i,j < n$ let $S_{i,j}=S+ (si/n) (1,0) + (sj/n)(1/2,\sqrt{3}/2)$.
We define $h_n$-fold colouring of the plane as follows. We assign colour $(i,j)$ to all points in hexagons fully contained in $S_{i,j}$. Now we will estimate the number $h_n$ of hexagons that are fully contained in one copy of $A$, let say $A^0$. Let $A^0_n$ be a rescaled copy o $A$ with the same center as $A^0$  and diameter $(1-2\frac{s}{n})$. The area of $A^0_n$ is $p_n= \frac{1}{4} (1-2\frac{s}{n})^2 (\pi - 6 \arcsin(x) + 6x\sqrt{1-x^2} )$, and the area of the hexagon is $(s/n)^2(\sqrt{3}/2)$. Observe that if a hexagon of the tiling has non-empty intersection with $A^0_n$ then it is fully contained in $A^0$. Therefore we can bound $h_n$ from below:
 $$h_n\geq\frac{p_n}{(s/n)^2(\sqrt{3}/2)}=\frac{  \frac{1}{4} (1-2\frac{s}{n})^2 (\pi - 6 \arcsin(x) + 6x\sqrt{1-x^2} )   }{(s/n)^2(\sqrt{3}/2)}=$$ $$=\sqrt{3}(\frac{n}{b+\sqrt{1-x^2}} -2)^2 (b x + x\sqrt{1-x^2}). $$ On the other hand we conclude that every hexagon is contained in $h_n$ of $S_{i,j}$ (since we can get all the hexagons by shifting one of them by $(si/n) (1,0) + (sj/n)(1/2,\sqrt{3}/2)$), hence every point in a hexagon has $h_n$ of $n^2$ colours assigned.\\
With $n$ tending to infinity the set of points coloured $1$ tends to $S$ and our sequence of $h_n$-fold colourings gives us an upper bound for the fractional chromatic number of $G_{[1,b]}$:

$$ \chi_f (G_{[1,b]})\leq \lim_{n\rightarrow \infty} \frac{n^2}{\frac{1}{2\sqrt{3}}\big(\frac{n}{b+\sqrt{1-x^2}} -2\big)^2 \big(\pi- 6\arcsin(x) + 6x\sqrt{1-x^2}\big)}=$$
$$ = \frac{2\sqrt{3}}{\pi- 6\arcsin(x) + 6x\sqrt{1-x^2}}  \cdot \lim_{n\rightarrow \infty} \frac{n^2}{(\frac{n}{b+\sqrt{1-x^2}} -2)^2}=$$ 
$$= \frac{2\sqrt{3} (b+\sqrt{1-x^2})^2}{\pi- 6\arcsin(x) + 6x\sqrt{1-x^2}}=$$
$$=\frac{2\sqrt{3} (b+\sqrt{1-x^2})^2}{\pi- 6\arcsin(x) + 3\sin(2\arcsin(x))}$$

Now lets explain why did we choose the ratio between the length of the circular arc - $y$ and the segment - $x$ to be $b$.
To find minimal value we take the derivative of the upper bound for $\chi_f (G_{[1,b]})$ presented above and check when it equals $0$.   
\begin{displaymath} \frac{4 \sqrt{3} x \left(b+\sqrt{1-x^2}\right) \left(6 b x+6 \arcsin(x)-\pi \right)}{\sqrt{1-x^2} \left(-6 \arcsin(x)+3
   \sin \left(2 \arcsin(x)\right)+\pi \right)^2} =0\end{displaymath}
\begin{displaymath}   
  \Longleftrightarrow 6 b x+6 \arcsin(x)-\pi =0 
\end{displaymath}

Applying equality \ref{xy} we obtain: 
\begin{displaymath}  
6 b x-6 y =0
\end{displaymath}
and hence
\begin{displaymath}   
   y=bx
\end{displaymath}


To complete the proof we transform the formula $\frac{2 \sqrt{3} \left(b+\sqrt{1-x^2}\right)^2}{\pi -6 \arcsin(x)+3
   \sin \left(2 \arcsin(x)\right) }$ using previous equality and by doing so we get:
   $\frac{\sqrt{3}}{3}\cdot\frac{b+\sqrt{1-x^2}}{x}$ as the upper bound for the fractional chromatic number of graph $G_{[1,b]}$.
\qed\end{proof}
Figure \ref{wykres} and Table \ref{tab0} below present the upper bound for fractional chromatic number of $G_{[1,b]}$ for different values $b$.

\begin{table}[h!]
\begin{center}
\begin{tabular}{c | c|c|c|c|c}
$b=$				 & 1    & 1,5  & 2   & 3     & 4 \\\hline
$\chi_f(G_{[1,b]})\leq$& 4,36 & 6,86 & 9,9 & 17,62 & 27,55\\
\end{tabular}
\end{center}\caption{Applications of Theorem \ref{twplacki}}\label{tab0}
\end{table}

\begin{figure}[h!]\label{wykres}
\center
\includegraphics[scale=1]{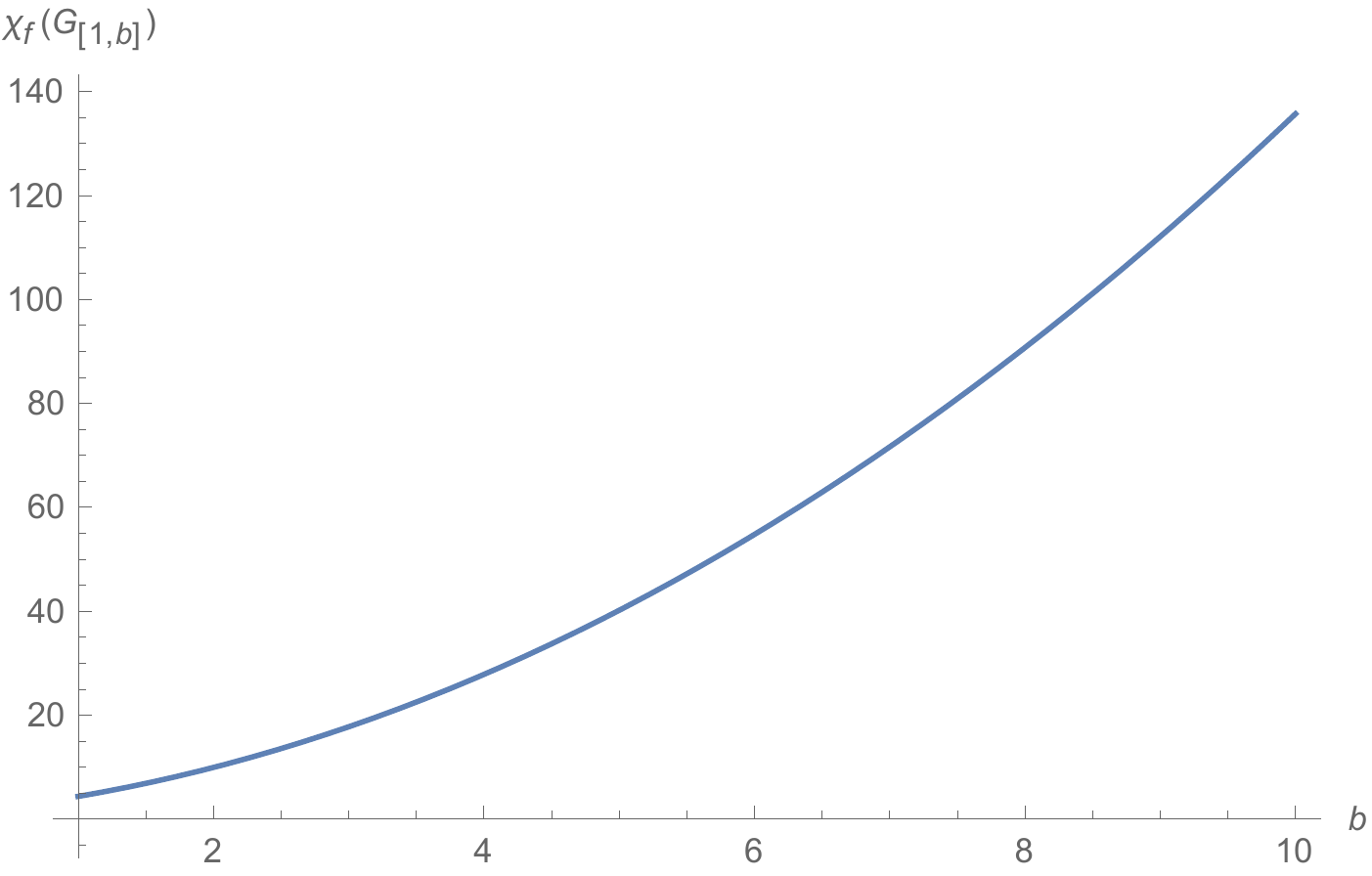}
\caption{The upper bound for $\chi_f(G_{[1,b]})$ from Theorem \ref{twplacki}}\label{wykres}
\end{figure}

\section{$j$-fold colouring of the plane}
Note that the upper bounds for fractional chromatic numbers of $G_{[1,b]}$ are established by presenting an infinite sequence of $j$-fold colourings. These colourings give results close to the upper bound only for very big $j$. However, as it was mentioned in the introduction, for the practical reasons it is often the case that only $j$-fold colouring with small $j$ are valuable to consider. Furthermore, this consideration leads also to a purely mathematical question: How fast, in terms of $j$, can we get close to the infinite limit - the upper bound for fractional chromatic number? In this section we give some insight in this matter.

\subsection{Methods for $G_{[1,1]}$}
In this subsection we concentrate on $G_{[1,1]}$.
\begin{theor}\label{tw2-3}
There are $2$-fold colouring of $G_{[1,1]}$ with $12$ colours and $3$-fold colouring of $G_{[1,1]}$ with $16$ colours.
\end{theor}
\begin{proof}
For our colourings we are going to use classic covering of the plane with hexagons with side length $1/2$. Obviously in a hexagon of side length $1/2$ there are pairs of point in a distance $1$, so when we say we colour a hexagon we mean the interior of the hexagon plus its right border and its three vertices: the upper one and the two on the right (see Figure \ref{hexagon}).
\begin{figure}[h]
\center\includegraphics[scale=0.5]{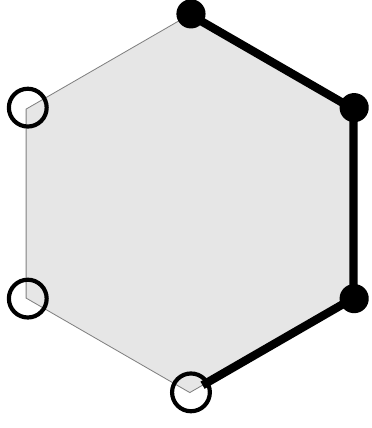}
\caption{Borders of a hexagon}\label{hexagon}
\end{figure}

Firstly we show our $2$-fold colouring of $G_{[1,1]}$ with 12 colours. We are going to use the hexagon grid twice with the set of colours $\{1,2,...,12\}$. We create first layer by simply giving each row of hexagons three numbers and use it periodically so that if one row has colours from set $\{1,2,3\}$ the next gets colours $\{4,5,6\}$ and so on. Having this colouring as our model we get second layer by moving the coloured grid by a vector $[3\sqrt{3}/4,-3/2]$. See Figure \ref{2fold} with first layer of colours and marked placement of colour $1$ in the layers. The  distances between two hexagons of the same color are: vertically $2$, horizontally $\sqrt{3}$ and diagonally $\frac{5\sqrt{3}}{8}\approx 1.08$, so they are all bigger than $1$.

\begin{figure}[h]
\center\includegraphics[scale=0.5]{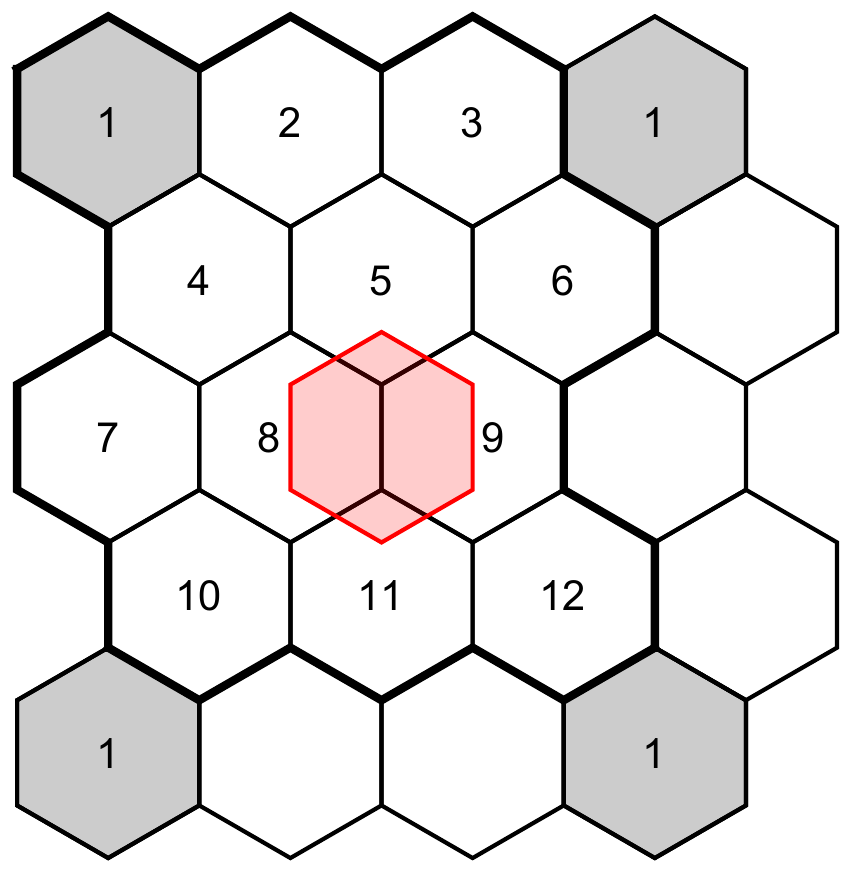}
\caption{$2$-fold colouring of $G_{[1,1]}$ with 12 colours}\label{2fold}
\end{figure}

To create a $3$-fold colouring of $G_{[1,1]}$ with $16$ colours we again start with the classic covering of the plane with hexagons with side length 1/2. We are going to use the hexagon grid 3 times with the set of colours $\{1,2,3,...,16\}$. We create first layer by simply giving each row of hexagons four numbers and use it periodically so that if one row has colours from set $\{1,2,3,4\}$ the next gets colours $\{5,6,7,8\}$ and so on. Having this colouring as our model we get second and third layers by moving the coloured grid by a vector $[\sqrt{3},-1]$. See Figure \ref{3fold} with first layer of colours and marked placement of colour $1$ in other layers. The distances between two hexagons of the same color are at least $1$ ($=1$ in case of two hexagons from different layers).
\begin{figure}[h]
\center\includegraphics[scale=0.5]{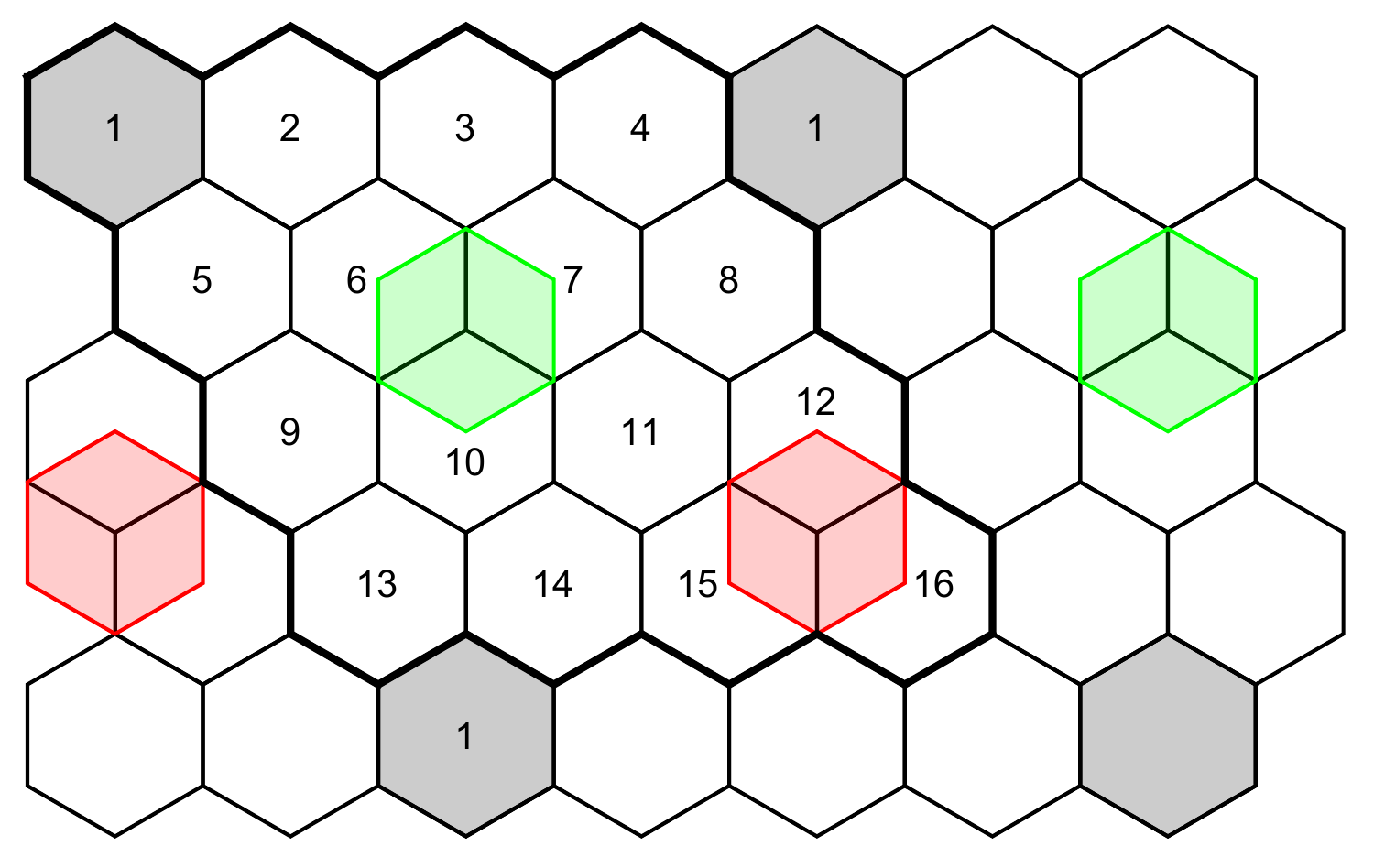}
\caption{$3$-fold colouring of $G_{[1,1]}$ with 16 colours}\label{3fold}
\end{figure}

\qed\end{proof}

\begin{theor}\label{tw37}

There is a $7$-fold colouring of $G_{[1,1]}$ with 37 colours i.e. $\frac{37}{7}\approx 5.285$.

\end{theor}

\begin{proof}
To create $7$-fold colouring of $G_{[1,1]}$ we start with a hexagon grid with hexagon side's length $s=\frac{1}{2\sqrt{7}}$. We colour with the first colour some hexagons in the pattern presented with blue colour on Figure \ref{7fold}. Then we shift our pattern by a vector $[\frac{\sqrt{3}}{2\sqrt{7}},0]$ and colour it with second colour (colour red on the Figure \ref{7fold}). Repeating this action 37 times we have 7 colours for each hexagon and there are 37 colours total. In our base colouring figure made by 7 hexagons the largest distance between two points is 1 so choosing half of the border to be coloured and the other not is enough to make sure there are no two point in distance 1 in one figure. The distance between two such figures is 
$$  \sqrt{      (3\sqrt{3}s)^2   +(2s)^2    } = s \sqrt{31} = \frac{\sqrt{31}}{2\sqrt{7}} \approx 1.05 $$
\begin{figure}[h]
\center\includegraphics[scale=0.5]{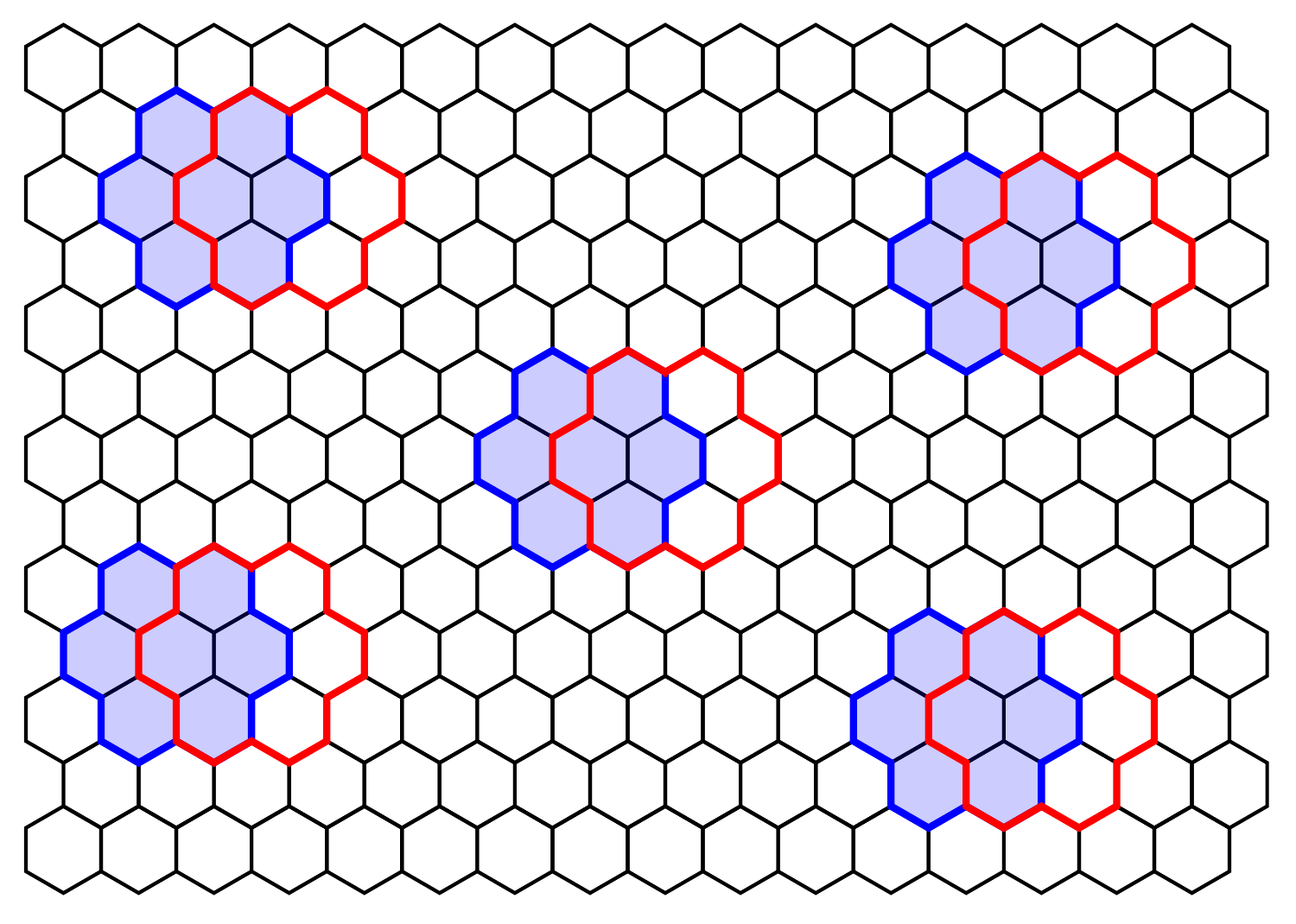}
\caption{$7$-fold colouring of $G_{[1,1]}$ with 37 colours}\label{7fold}
\end{figure}
\qed\end{proof}

\subsection{Methods for $G_{[a,b]}$}
In this section we give two general methods for building $j$-fold colourings (for small $j$) for graphs $G_{[1,b]}$.
\begin{theor} \label{nm}
There exists a $nm$-fold colouring with $\lceil(\frac{2b}{\sqrt{3}}+1)\cdot n\rceil\cdot \lceil(\frac{2b}{\sqrt{3}}+1)\cdot m\rceil$ colours of the graph $G_{[1,b]}$ i.e.
 $\frac{\chi_{nm}(G_{[1,b]})}{nm}\leq \frac{\lceil(2b/ \sqrt{3}\ +1)\cdot n\rceil\cdot \lceil(2b/ \sqrt{3}\ +1)\cdot m\rceil}{nm}$.
\end{theor}
\begin{proof}
In the proof we are going to create $n\cdot m$ coloured hexagon grids. A colour of a hexagon will be a pair of numbers, first of which will be related to the row the hexagon is in and the second corresponds to the column.\\
Let $W_1^1$ be a hexagon grid with hexagons with side length 1/2. Let $H$ be one of the hexagons from $W_1^1$. For $2\leq j\leq n$ let $W_1^j$ be a hexagon grid created by moving uncoloured $W_1^1$ by a vector $(j-1)/n\ [\sqrt{3}/2,0]$. Now lets say $H$ is coloured $(1,1)$ and let's find first hexagon $H'$ in the same row from any $W_1^j$ that can also be coloured $(1,1)$ without creating a monochromatic edge in $G_{[1,b]}$ i.e at distance greater or equal to $b$ (see Figure \ref{nm1}). 
\begin{figure}[h]
\includegraphics[width=7cm]{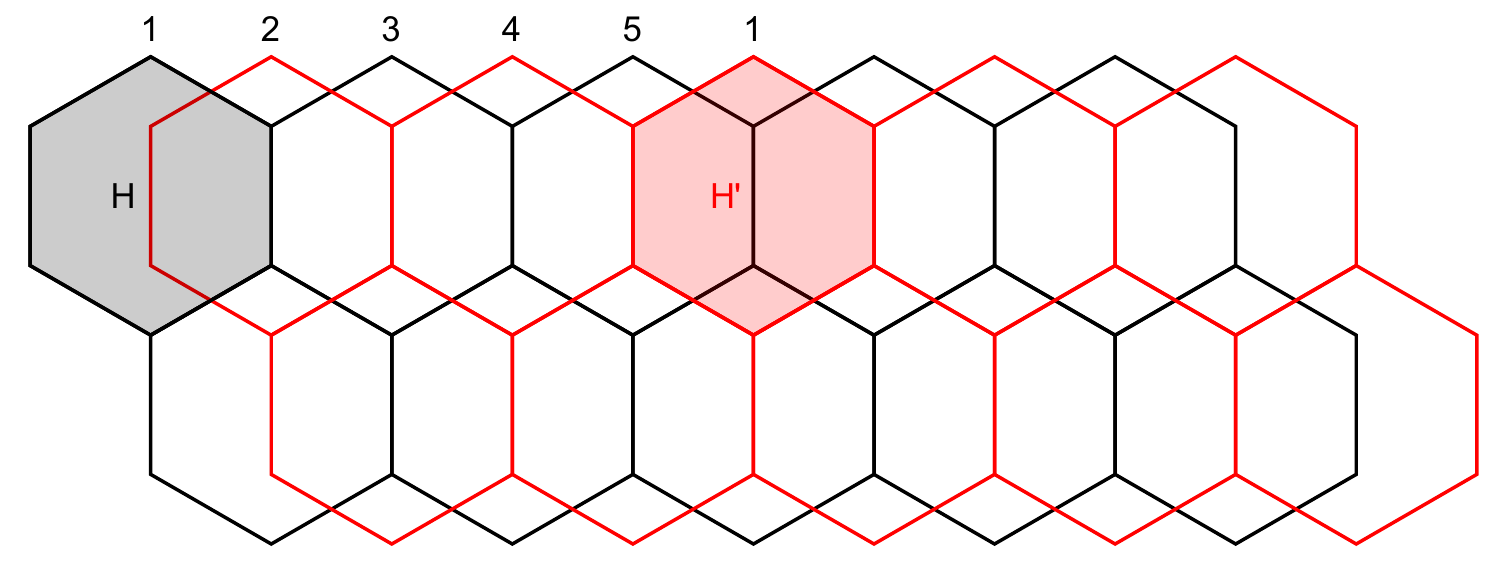}
\caption{2 folds in 4-fold colouring $G_{[1,1]}$ with 25 colours}\label{nm1}
\end{figure}

Counting the hexagons from $H$ to $H'$ we find that there are exactly $\lceil(b+\frac{\sqrt{3}}{2})\frac{2 n}{\sqrt{3}}\rceil=\lceil(\frac{2b}{\sqrt{3}}+1)\cdot n\rceil$ of them, since the distance between the centers of $H$ and $H'$ has to be at least $b+\frac{\sqrt{3}}{2}$. Then we shift each of $W_1^j$ by vectors $(i-1)/m\ [\sqrt{3}/4,-3/4]$ for $2\leq i\leq m$ getting grids $W_i^j$. Now remembering that $H$ is coloured $(1,1)$ we find the first hexagon $H''$ from $W_i^1$ to have the same colour (see Figure \ref{nm2}). The number of hexagons between the two is $\lceil(2b/ \sqrt{3}\ +1)\cdot m\rceil$.

\begin{figure}[h]
\includegraphics[width=9cm]{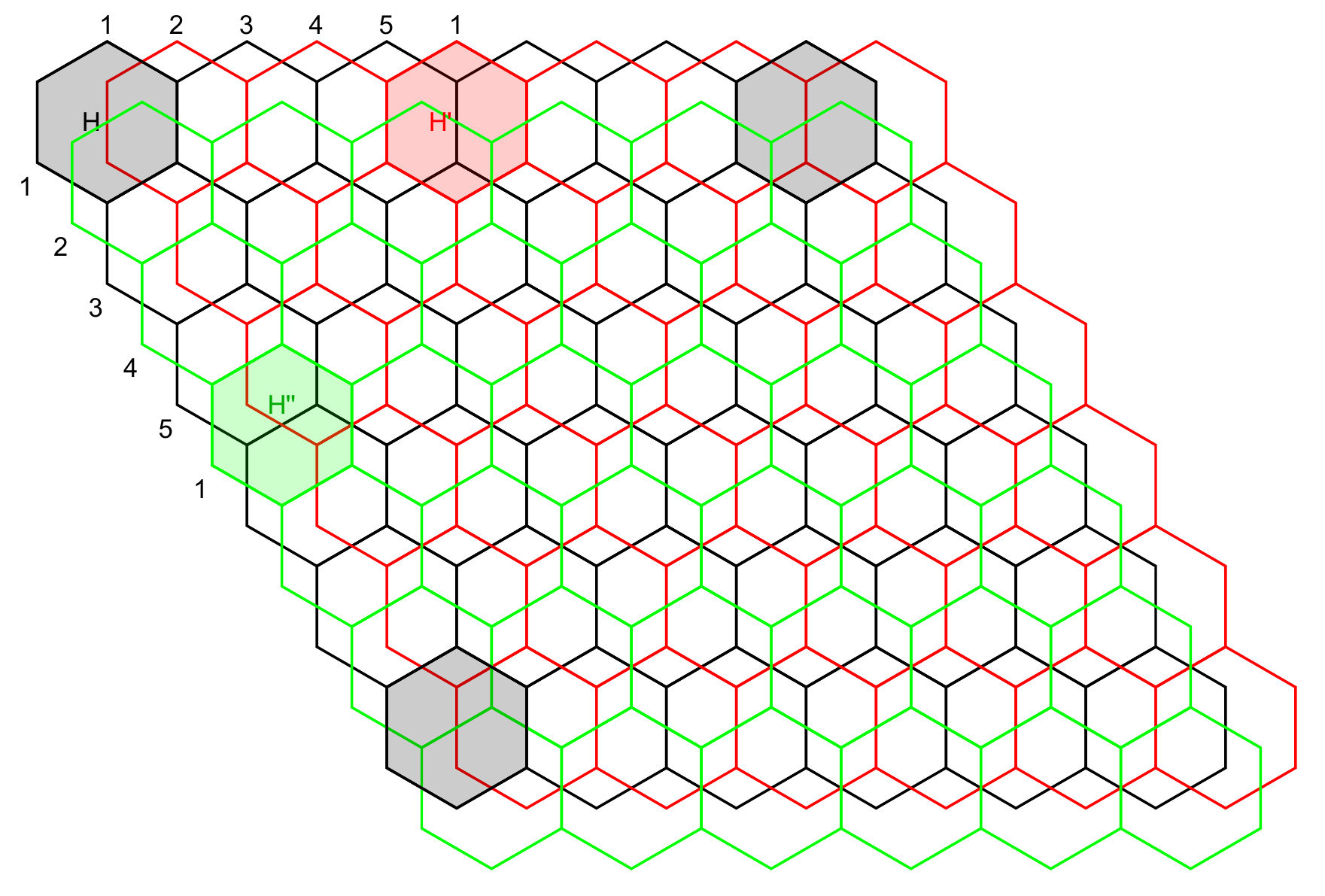}
\caption{3 folds in 4-fold colouring $G_{[1,1]}$ with 25 colours}\label{nm2}
\end{figure}

Now we have three hexagons $H$, $H'$, and $H''$ that can be coloured with the same colour, since the distance between $H'$ and $H''$ would be the smallest when $n=m$ and then the centers of the three hexagons create an equilateral triangle with side length at least $b+ \frac{\sqrt{3}}{2}$ (because $H$ and $H'$ are in proper distance).\\
The final colouring is created by giving each hexagon a pair of numbers, so that $H$ gets colour $(1,1)$, the next hexagon in the same row gets $(1,2)$ and so on until $H'$ gets $(1,1)$ again and the cycle repeats, and in the row below we get colours $(2,1),\ (2,2),\ (2,3)...$ and so on until we have row with $H''$ in which we use $(1,1),\ (1,2),\ (1,3)...$ again. Since we have $nm$ hexagon grids every point gets $nm$ colours. Finaly we get $nm$-fold colouring of $G_{[1,b]}$ with $\lceil(\frac{2b}{\sqrt{3}}+1)\cdot n\rceil\cdot \lceil(\frac{2b}{\sqrt{3}}+1)\cdot m\rceil$ colours.

\begin{figure}[h]
\includegraphics[width=9cm]{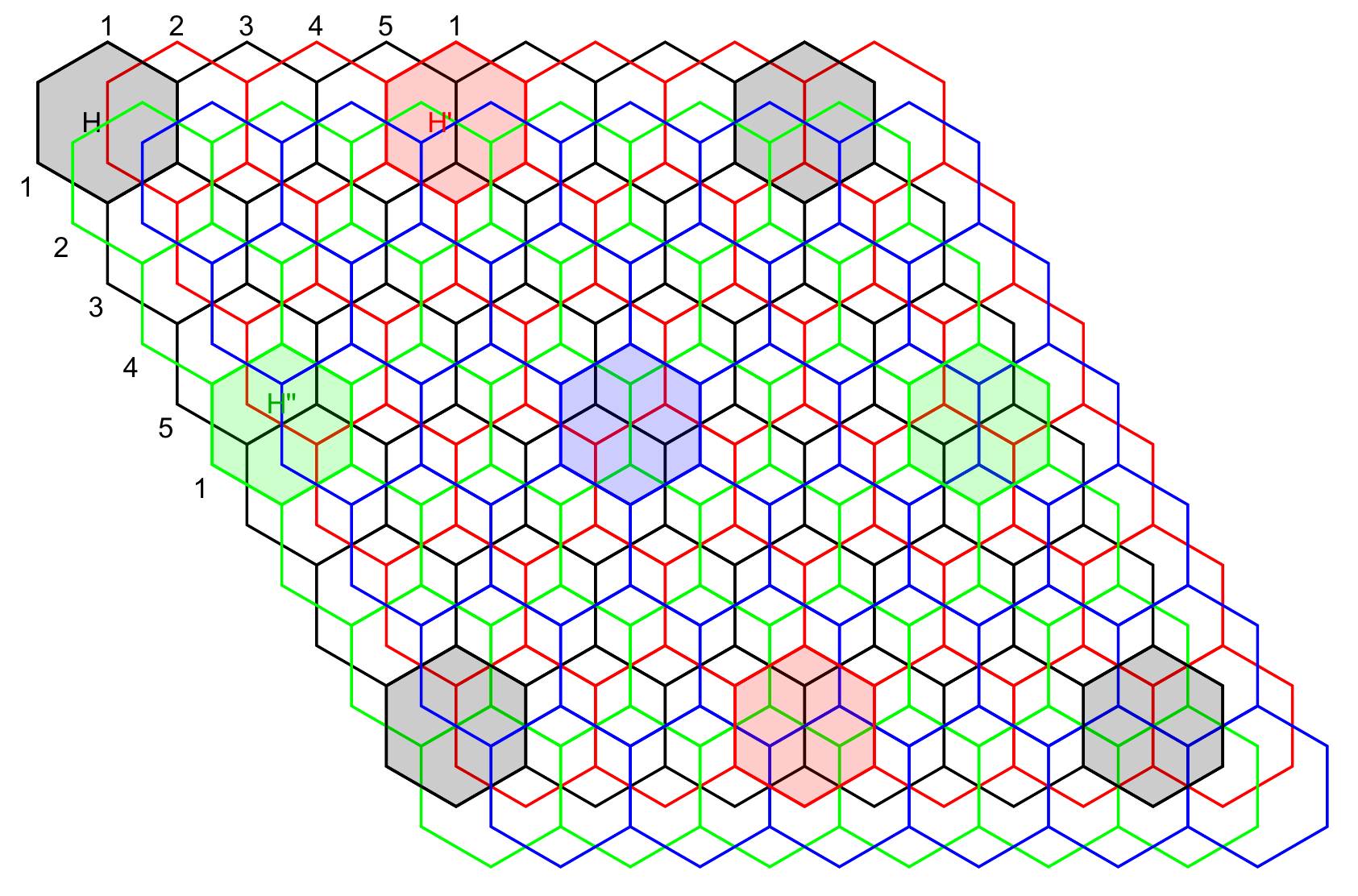}
\caption{4-fold colouring $G_{[1,1]}$ with 25 colours}\label{nm3}
\end{figure}
\qed\end{proof}

The following result can be seen as a combination of $2$-fold colouring approach from Theorem \ref{tw2-3} and the main method from Theorem \ref{nm}.

\begin{theor}\label{2nm}
There exists a $2nm$-fold colouring with $2\lceil(b+1)2n\rceil\lceil(b+1)\frac{2m}{3}\rceil$ colours of the graph $G_{[1,b]}$ i.e.
 $\frac{\chi_{2nm}(G_{[1,b]})}{2nm}\leq \frac{2\lceil\sqrt{3}(b+1)\frac{2n}{3}\rceil\lceil(b+1)\frac{2m}{3}\rceil}{2nm}$.
\end{theor}

\begin{proof}
Note that the proof uses similar methods to the previous one. We are going to create $2\cdot n\cdot m$ coloured hexagon grids. A colour of a hexagon will be a pair of numbers, first of which will be related to the row the hexagon is in and the second corresponds to the column.\\
Let $W_1^1$ be a hexagon grid with hexagons with side length 1/2. Let $H$ be one of the hexagons from $W_1^1$. For $2\leq i\leq m$ let $W_i^1$ be a hexagon grid created by moving uncoloured $W_1^1$ by a vector $(i-1)/m\ [0,-3/2]$. Now lets say $H$ is coloured $(1,1)$ and let's find first hexagon $H'$ in the same column from any $W_i^1$ that can also be coloured $(1,1)$ without creating a monochromatic edge in $G_{[1,b]}$ i.e at distance greater or equal to $b$, so the distance between the centers of $H$ and $H'$ needs to be at least $b+1$ (see Figure \ref{j212} for $4$-fold colouring with $4\cdot 6=24$ colours of $G_{[1,1]}$ with $n=1$ and $m=2$)
\begin{figure}[!h]
\includegraphics[width=9cm]{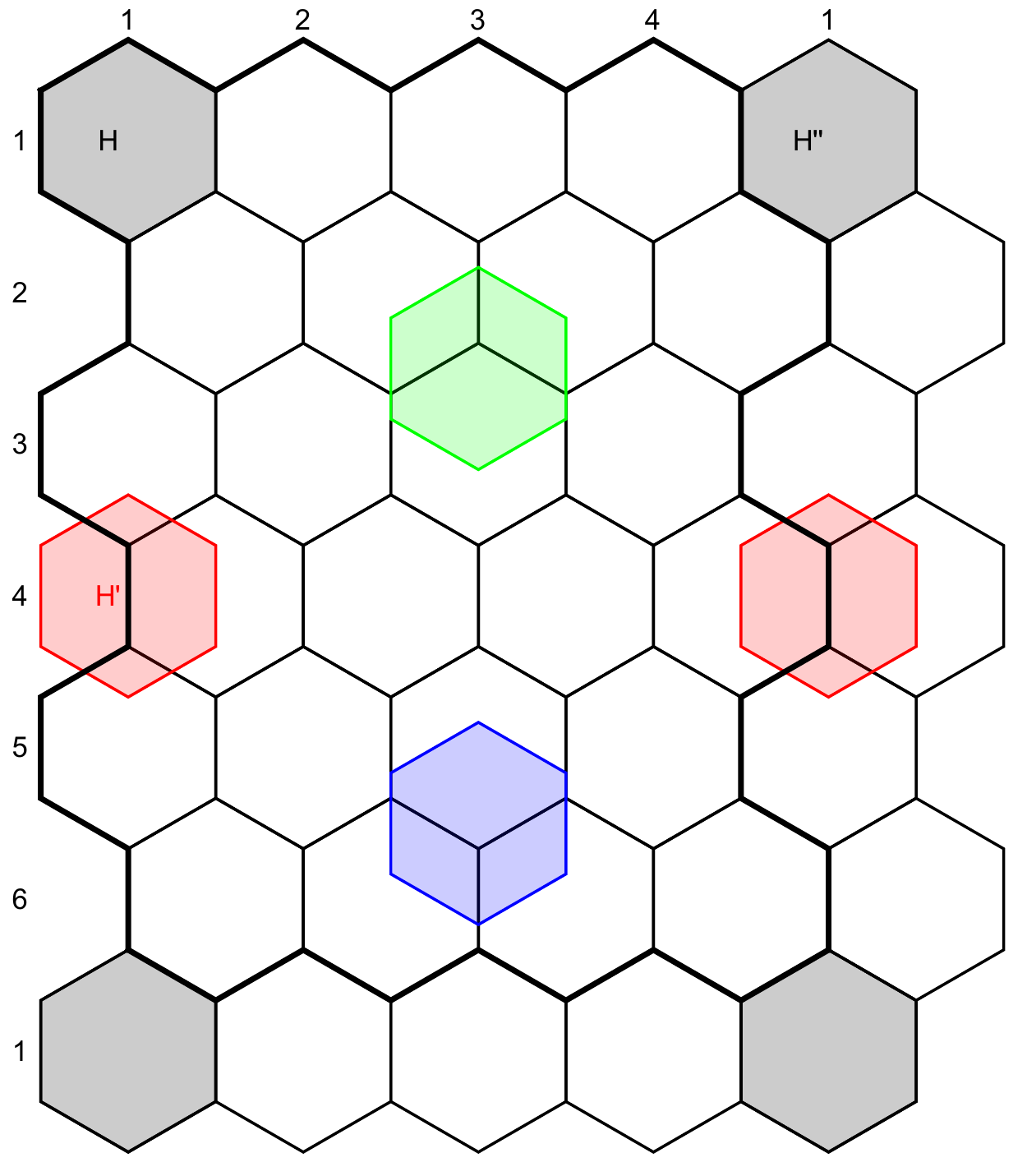}
\caption{4-fold colouring $G_{[1,1]}$ with 24 colours}\label{j212}
\end{figure}
Since the distance between the centers of $H$ and $H'$ is at least $b+1$ but we chose $H'$ to be as close as possible to $H$, then there are $\lceil(b+1)\frac{2m}{3}\rceil$ rows between $H$ and $H'$. Now for $2\leq j\leq m$ and $2\leq j\leq n$ let $W_i^j$ be a hexagon grid created by shifting $W_i^1$ by a vector $(j-1)/n\ [\sqrt{3}/2,0]$. Let $H''$ be the first hexagon in the same row as $H$ in any of $W_1^j$ such that the distance between the centers of the two is at least $\sqrt{3}(b+1)$. The number of hexagons between $H$ and $H''$ is  $\lceil\sqrt{3}(b+1)\frac{2n}{\sqrt{3}}\rceil= \lceil (b+1)2n\rceil\geq b+\frac{\sqrt{3}}{2}$.\\
The distance between the centers of $H'$ and $H''$ is at least: 
$$\sqrt{ \big(\ \big\lceil(b+1)\frac{2m}{3}\big\rceil\ \frac{3}{2m\ }\big)^2 \cdot  \big(\ \big\lceil\sqrt{3}(b+1)\frac{2n}{\sqrt{3}}\big\rceil\ \frac{\sqrt{3}}{2n}\ \big)^2 }\geq$$ $$ \geq \sqrt{(b+1)^2 + (\sqrt{3}(b+1))^2}= \sqrt{4(b+1)^2} = 2b+2.$$
Since the distance is at least $2b+2$ they can be coloured the same colour and there is enough space between them to put another hexagon in the same colour between them. So we create new hexagon grids $V_i^j$ by shifting $W_i^j$ by a vector $ [ \big\lceil\sqrt{3}(b+1)\frac{2n}{\sqrt{3}}\big\rceil \cdot \frac{\sqrt{3}}{4n}, \big\lceil(b+1)\frac{2m}{3}\big\rceil \cdot \frac{3}{4m} ] $.
The final colouring is created by giving each hexagon a pair of numbers, so that $H$ gets colour $(1,1)$ the next hexagon in the same row from any $W_1^j$ gets $(1,2)$ and so on until $H''$ gets $(1,1)$ again and the cycle repeats. The next row from $W_1^j$ has the first coordinate in all colours equal to $2$. The set of numbers for second coordinate is $\{1,2,..., \big\lceil\sqrt{3}(b+1)\frac{2n}{\sqrt{3}}\big\rceil\}$, and for the first coordinate $2\big\lceil(b+1)\frac{2m}{3}\big\rceil$ (since there are 2 rows of hexagons between 2 hexagons from the same grid in the same column). So we use $ 2\big\lceil(b+1)\frac{2m}{3}\big\rceil \cdot  \big\lceil\sqrt{3}(b+1)\frac{2n}{\sqrt{3}}\big\rceil $ colours using $2nm$ hexagon grids.
\qed\end{proof}

\subsection{Summary}

Theorems \ref{nm} and \ref{2nm} give different results and we cannot say one is stronger than another. Table \ref{tab1} presents comparison of the two for $j$-fold colouring of $G_{[1,1]}$ with small values of $j$ ($k$ is the numbers of colours used). We bold best results for a fixed $j$.

\begin{table}[!h]
\begin{center}
\begin{tabular}{c||c | c|c|c|c|c|c}
method: & $j=$			 & 2    & 4   & 6     & 8 & 10& 12 \\\hline
j=nm &  $k=$     &\textbf{15}    & 25  & 35   &\textbf{45}   &\textbf{55}  & \textbf{63}\\
 & $k/j \approx$ & 7,5  &6,25 &5,83  &5,63 &5,5 & 5,25\\\hline
j=2nm & $k=$     &   16 &\textbf{24}   &\textbf{32}    &48 &56  &64     \\
&$k/j \approx$   & 8    &6    & 5.33 &6  &5,6 & 5,33
\end{tabular}\caption{Applications of Theorems \ref{nm} and \ref{2nm} for $G_{[1,1]}$ with small $j$.}\label{tab1}
\end{center}
\end{table}

Our best results in $j$-fold colouring of $G_0$ with small values of $j$ using $k$ colours are summarized in Table \ref{tab2}. They follow from Theorems \ref{tw2-3}, \ref{tw37}, \ref{nm}, \ref{2nm}.
\begin{table}[!h]
\begin{center}
\begin{tabular}{c r|c|c|c|c|c|c|c}
$k$&$=$& 7 & 12 & 16 & 24& 32 & 32 & 37\\\hline
$j$&$=$& 1 & 2 & 3 & 4 & 5& 6& 7\\\hline
$k/j$&$ \approx$& 7 & 6 & 5.33 & 6 & 6.6& 5.33& 5.26
\end{tabular}\caption{Selected results in $j$-fold colouring of $G_{[1,1]}$ with small $j$.}\label{tab2}
\end{center}
\end{table}

For practical applications it is useful to consider $j$-fold colouring of graph $G_{[1,2]}$, especially with small $j$. Table \ref{tab3} presents our results in colouring $G_{[1,2]}$ using method from Theorem \ref{nm}. It appears that in this case the method from Theorem \ref{2nm} does not give good results.

\begin{table}[!h]
\begin{center}
\begin{tabular}{c | c|c|c|c|c}
 $k=$   &  12&  70& 100& 930 &960 \\\hline
 $j=$			 & 1    & 6   & 9     & 84 & 87 \\\hline
 $k/j \approx$ & 12 &  11,67 &11,11& 11,07 & 11,03\\\hline
\end{tabular}\caption{Applications of Theorem \ref{nm} for $G_{[1,2]}$ with small $j$.}\label{tab3}
\end{center}
\end{table}

\section*{acknowledgement}
We thank professor Zbigniew Lonc, Zbigniew Walczak and profesor Jacek Wojciechowski for introducing us the problem. 




\end{document}